\documentclass[12pt,electronic]{amsart}                                                 
\usepackage[T1]{fontenc}
\usepackage[applemac]{inputenc}
\usepackage[english]{babel}
\usepackage[small,nohug]{diagrams}
\usepackage[hmargin=3.7cm,vmargin=3.5cm]{geometry} 
\usepackage{graphicx}                         
\usepackage{amsmath}
\usepackage{amsthm}                      
\usepackage{amssymb}
\usepackage{amscd}     
\usepackage{mathrsfs}
\usepackage{stmaryrd}
\usepackage{enumerate}
\usepackage[small,it]{caption}
\theoremstyle{plain}                                                           
\newtheorem{thm}{Theorem}[section]

\newtheorem{prop}[thm]{Proposition}

\theoremstyle{definition}

\newtheorem{rem}[thm]{Remark}
\newtheorem{defn}[thm]{Definition}

\DeclareMathOperator{\Hom}{Hom}
\DeclareMathOperator{\SL}{SL}

\DeclareMathOperator{\Spec}{Spec}
\DeclareMathOperator{\Sym}{Sym}                  

\DeclareMathOperator{\Pic}{Pic}

\DeclareMathOperator{\Aut}{Aut}                          

\DeclareMathOperator{\sgn}{sgn}

\newcommand{\field}[1]{\ensuremath{\mathbf{#1}}}
        
\newcommand{\Q}{\ensuremath{\field{Q}}}        

\newcommand{\sym}{\ensuremath{\mathbb{S}}}
\newcommand{\Z}{\ensuremath{\field{Z}}} 

\newcommand{\p}{\ensuremath{\field{P}}}

\newcommand{\iso}{\cong}
\newcommand{\M}{\mathcal{M}}

\newcommand{\MM}{\overline{\mathcal{M}}}

\usepackage{enumerate}
\usepackage{wasysym}
\title{Cusp form motives and admissible $G$-covers}
\author{Dan Petersen}
\newcommand{\X}{\mathscr{X}}
\newcommand{\ri}{\overline{\mathcal{I}}}
\newcommand{\Zm}{\field{Z}/m\field{Z}}
\newcommand{\Zp}{\field{Z}/p\field{Z}}
\newcommand{\sm}{\ensuremath{\mathrm{sm}}}
\newcommand{\Gm}{\ensuremath{\field{G}_m}}

\renewcommand{\sym}{\mathbb S}
\newcommand{\dr}{\mathrm{DR}}

\begin{document} 
  
 \maketitle

\begin{abstract} The moduli space $\MM_{1,n}(B(\Zm)^2)$ of twisted stable maps into the stack $B(\Zm)^2$ carries a natural $\sym_n$-action and so its cohomology may be decomposed into irreducible $\sym_n$-representations. Working over $\Spec \Z[1/m]$ we show that the alternating part of the cohomology of one of its connected components is exactly the cohomology associated to cusp forms for $\Gamma(m)$. In particular this offers an alternative to Scholl's construction of the Chow motive associated to such cusp forms. This answers in the affirmative a question of Manin on whether one can replace the Kuga-Sato varieties used by Scholl with some moduli space of pointed stable curves. \end{abstract}

\section{Introduction}
\newcommand{\Jac}{\Pic^0}

Deligne showed how one can associate a compatible system of $\ell$-adic Galois representations for almost all $\ell$ to elliptic modular forms in \cite{deligne69}. Taking these ideas further, one may wish to find a motive associated to modular forms, in the sense that the different $\ell$-adic realizations of the motive are exactly the associated Galois representations. This idea was carried out for elliptic modular forms by \cite{scholl90}. See also \cite{blasius} where Chow motives are associated to Hilbert modular forms.

The Galois representations associated to modular forms can often be realized as subquotients of the cohomology of a smooth projective variety $X$. To show that these subquotients are actually realizations of a motive, one needs to construct a suitable idempotent correspondence in
\[ A^{\dim X}(X \times X)\]
which ``cuts out'' these pieces of the cohomology. The standard conjectures \cite{kleimanstandardconjectures} would imply that such correspondences exist in great generality. However, assuming that the standard conjectures will not be proven in the near future, it is still interesting to construct such correspondences by ad hoc methods in cases of interest. 

Consider classical elliptic modular forms. Let $\Gamma(m)$ denote the level $m$ principal congruence subgroup of $\SL(2,\Z)$. The space of modular forms for $\Gamma(m)$ of weight $n+2$ can be found as a subquotient of the cohomology of the $n$th fibered power of the universal elliptic curve (a Kuga-Sato variety) over the modular curve $Y(m)$. Specifically, one can find in the Betti cohomology the direct sum of the space of holomorphic cusp forms and its complex conjugate, the space of antiholomorphic cusp forms. The corresponding parts of the $\ell$-adic cohomology of the variety are the Galois representations associated to the modular forms. 

To assemble these realizations into a Chow motive one would need first to find a smooth compactification of the $n$th fibered power of the universal curve. To compactify one can take fibered powers of the universal generalized elliptic curve (in the sense of \cite{dr73}) over $X(m)$; this compactification is however singular at the boundary. A desingularization of the boundary was constructed by Deligne \cite{deligne69}, and Scholl showed in \cite{scholl90} that the projector given by the ''alternating'' representation of the hyperoctahedral group
\[ \sym_2 \wr \sym_n = (\underbrace{\sym_2 \times \ldots \times \sym_2}_{n \text{ times}}) \rtimes \sym_n\]
(which canonically acts on the $n$th fibered power and its desingularization) cuts out exactly the cohomology coming from cusp forms. Thus there exists a Chow motive associated to cusp forms of given weight for $\Gamma(m)$. 

From the point of view of moduli theory, the best possible way to construct a smooth compactification of a moduli problem is to change the definition of the moduli problem to allow also some appropriate degenerations of the objects one is parametrizing. In this sense Deligne's desingularization is not so natural. Taking the $n$th fibered power gives you a moduli space of elliptic curves equipped with some level structure, with $n+1$ marked points which are allowed to coincide. This suggests the possibility of constructing a smooth compactification by considering pointed stable curves of genus one with some kind of level structure. This suggestion was put forth by Manin in \cite{iteratedintegrals} and \cite{iteratedshimura}. Not only would one then have a modular interpretation of the points on the boundary, but also a space with far more structure: for instance, stable curves have a well understood deformation theory, and they form an operad. (The operadic point of view of stable curves also gets used in this article when computing the contribution from the boundary.)

In level one this compactification would just be $\MM_{1,n+1}$. This space is a smooth and proper stack over $\Z$ and its cohomology defines a Chow motive. In \cite{faberconsani} it is shown, independently of Manin's question, that the projector given by the alternating representation of $\sym_{n+1}$ acting on $\MM_{1,n+1}$ cuts out exactly the space of cusp forms for $\SL(2,\Z)$ of weight $n+2$. Thus an alternative to Scholl's construction is found, which however is only valid for modular forms of level $1$ due to a lack of an analogue of the space $\MM_{1,n+1}$ for curves with level structure. 

Fix a positive integer $m$ and let us always assume that $m$ is invertible on our base scheme. By a level $m$ structure on a smooth curve $C$, we mean the choice of an isomorphism 
\[ (\Z/m\Z)^{2g(C)} \cong \Jac(C)[m]. \]
To extend Consani and Faber's construction to higher levels, one would need a smooth and projective moduli space of pointed stable curves with level structure. The problem of constructing such a compactification has a long history. A moduli space $\M_{g,n}(m)$ of pointed \emph{smooth} curves with level structure is easy to construct, in particular as it is a scheme for $m \geq 3$. On the boundary, one runs into problems when trying to define a good notion of a level structure on curves not of compact type, i.e.\ curves whose Jacobian is an extension of an abelian variety by a torus. The problem is roughly that these curves have ``too little'' $m$-torsion: for a torus $T$, the $m$-torsion group $T[m]$ has order $m^{\dim T}$, while for an abelian variety $A$ the $m$-torsion group has order $m^{2\dim A}$. 

When $g=1$, $n=1$ a good modular compactification is described in \cite{dr73}. In \cite{acv03} a smooth proper stack compactifying $\M_{g,n}(m)$ for any $g$ and $n$ was constructed, seemingly as a byproduct of the work of Abramovich, Vistoli and others on constructing a smooth and projective moduli space of stable maps into a stack. In Section 3 we construct an explicit isomorphism between the stack defined by Deligne and Rapoport and the one defined by Abramovich, Corti and Vistoli when $g=1$ and $n=1$. Although this is perhaps not so surprising, the author has not seen this observed anywhere in the literature, and the isomorphism is a bit striking.

Once we are armed with a smooth and proper moduli space $\MM_{g,n}(m)$ of pointed stable curves with level $m$ structure, we show in sections 4 and 5 of this paper that Consani and Faber's construction carries over to this setting as well: the pair of $\MM_{1,n}(m)$ and the projector given by the alternating representation of $\sym_n$ defines the Chow motive of cusp forms of weight $n+1$ for $\Gamma(m)$. The alternating part of the cohomology of the open part $\M_{1,n}(m)$ is exactly the cohomology coming from \emph{all} modular forms, i.e.\ both Eisenstein series and cusp forms. The alternating part of the cohomology of the boundary is isomorphic to the space of Eisenstein series only, and when one computes the cohomology of the total space $\MM_{1,n}(m)$ the Eisenstein series ``cancel'' exactly, leaving only the contribution from cusp forms.

Finally in section 6 of the paper we show that in this set-up, the Hecke correspondences can be given a modular interpretation over the boundary as well. 

I am grateful to my advisor Carel Faber for patient discussions. I have also benefitted from useful conversations with Nicola Pagani and Sergey Shadrin.


\section{Background}
\subsection{Twisted stable maps, admissible covers and level structures}

\newcommand{\bal}{\mathrm{bal}}

In \cite{av02} Abramovich and Vistoli introduced a stack generalizing the Kontsevich space $\MM_{g,n}(X)$ of stable maps to a projective variety, namely the stack $\MM_{g,n}(\X)$ of so called \emph{twisted} stable maps into a tame Deligne-Mumford stack $\X$ with projective coarse moduli space. To this end Abramovich and Vistoli first define the notion of a pointed twisted curve: this is a DM-stack $\mathcal{C}$ whose coarse moduli space $C$ is a pointed nodal curve, with the property that $\mathcal{C}\to C$ is an isomorphism away from the nodes and the marked points, with specific restrictions on the type of ''stacky'' structure $\mathcal{C}$ may have. Roughly speaking, the fibers of $\mathcal{C} \to C$ should all be cyclotomic gerbes. See \cite{av02} for a precise definition. With this definition in place, there is a proper moduli stack $\MM_{g,n}(\X)$ parametrising flat families of twisted curves equipped with a representable map to $\X$, such that the induced map of coarse moduli spaces $C \to X$ is stable in the sense of Kontsevich. 

There is an important open and closed substack  $\MM_{g,n}^\bal (\X)$ of $\MM_{g,n}(\X)$ which parametrizes twisted stable maps where the source curve is \emph{balanced}. These are exactly the twisted curves that are smoothable, i.e. which can be written as stable limits of smooth curves.

Let $G$ be a finite group. We assume that $|G|$ is invertible on our base scheme, so the classifying stack $BG$ is tame. The specific stack $\MM_{g,n}^\bal(BG)$ has an alternative description in terms of admissible covers, as explained in \cite{acv03}. Let us make the following definition:

\begin{defn} Let $C$ be a pointed nodal curve and $G$ a finite group. An \emph{admissible torsor} for $G$ over $C$ consists of a morphism of curves $P \to C$ and an action of $G$ on $P$, such that
\begin{enumerate}
\item the curve $C$ is identified with the scheme quotient $P/G$;
\item the map $P \to C$ is an admissible cover;
\item the restriction of $P \to C$ away from the nodes and markings of $C$, with the given $G$-action, is a torsor for the group $G$. 
\end{enumerate}\end{defn}

Then \cite{acv03} shows that giving a representable morphism from a balanced twisted curve $\mathcal{C}$ to $BG$ is canonically the same as giving an admissible $G$-torsor over the coarse moduli space of $\mathcal{C}$: given $\mathcal{C}\to BG$ one gets a $G$-torsor over $\mathcal{C}$ whose total space $P$ is a nodal curve, and the composition $P\to \mathcal{C}\to C$ is an admissible torsor; conversely, given an admissible torsor $P\to C$, the stack quotient $[P/G]$ is a balanced twisted curve.

Since perhaps most readers are more comfortable with admissible covers than 
with twisted curves, we shall stick to the language of admissible covers as far as possible in this article. 

As mentioned in the introduction, the existence of the stack $\MM_{g,n}(BG)$ allows one to give a modular compactification of the space $\M_{g,n}(m)$ of smooth pointed curves with level structure. Let $G = (\Zm)^{2g}$. In particular, we assume $m$ is invertible. Suppose given a smooth curve $C$ and a level $m$ structure on $C$, that is, a not necessarily symplectic isomorphism $\Jac (C)[m] \cong G$. Let $\pi_1(C)$ denote the \'etale fundamental group of $C$. Since the points of $\Jac(C)[m]$ correspond to cyclic \'etale covers of degree $m$ of $C$, one gets an isomorphism
\[ \Jac(C)[m] \cong \pi_1(C)/m \pi_1(C),\]
and hence a bijection between isomorphisms $\Jac(C)[m] \cong G$ and surjective morphisms $\pi_1(C) \to G$, i.e. connected $G$-torsors over $C$. See \cite[XIII, 2.12]{sga1}. So the level structure induces a representable map $C \to BG$ and when $2-2g-n < 0$ we find a morphism
\[ \M_{g,n}(m) \to \MM_{g,n}(BG).\] This morphism is not an open immersion. The problem is that every point of $\MM_{g,n}(BG)$ has $G$ in its automorphism group, coming from the automorphisms of the admissible torsors, while a level structure should generally have no automorphisms. This defect is fixed by composing with the rigidification map
\[ \MM_{g,n}(BG) \to \MM_{g,n}(BG)\!\!\! \fatslash G\]
(where we follow the notation of \cite{romagny05}), in the sense that the composed map is an isomorphism onto an open substack. The closure of $\M_{g,n}(m)$ in $\MM_{g,n}(BG)\!\!\! \fatslash G$ is the desired compactification. 

The stacks $\MM_{g,n}(BG)$ and  $\MM_{g,n}(BG)\!\!\! \fatslash G$ share the same coarse moduli space \cite[Theorem 5.1.5]{acv03} and in particular they have the same rational and $\ell$-adic cohomology. Since in this article we shall only be interested in their cohomology, we propose to \emph{ignore the process of rigidification in this article} (except in Section 3) as it would mostly be a small nuisance. 

Thus we define $\MM_{g,n}(m)$ to be the closure of the image of $\M_{g,n}(m)$ already in $\MM_{g,n}(BG)$. One can describe this closure explicitly: it is the open and closed substack of $\MM_{g,n}(BG)$ consisting of connected admissible $G$-torsors which are unramified over each marked point. Then $\MM_{g,n}(m)$ is a smooth \cite[Theorem 3.0.2]{acv03} and proper DM-stack over $\Spec \Z[1/m]$. 

Let $\zeta_m$ be a primitive $m$th root of unity. The stack $\MM_{g,n}(m)$ has $\phi(m) = |U(\Z/m\Z)|$ components, each of which is defined over $\Spec \Z[1/m,\zeta_m]$ and which are permuted by $\mathrm{Gal}(\Q(\zeta_m)/\Q)$. Any one of these components may be taken as the moduli space of curves with \emph{symplectic} level $m$ structure.

\section{Comparison of DR and ACV moduli stacks}
\label{comparesection}
\label{comparison1}

Fix an integer $m$. In \cite{dr73} a moduli stack parametrising generalised elliptic curves with full level $m$ structure was defined over $\Spec \Z[1/m]$. Let us denote it $\MM_\dr(m)$ in this section. Just as in the case of $\MM_{g,n}^{}(m)$ it consists of $\phi(m)$ open and closed substacks, all of which are individually defined over $\Spec \Z[1/m,\zeta_m]$. A choice of a primitive root $\zeta_m$ lets us identify one of these components with the modular curve $X(m)$ parametrising elliptic curves with symplectic level $m$ structure. 

In this section we show that the stack $\MM_{1,1}(m)$ is isomorphic to $\MM_\dr(m)$, provided that one includes the rigidification procedure as in \cite{acv03}. The isomorphism is quite simple: one finds that when $P \to C$ is an admissible $G$-torsor,  the curve $P$ is in a canonical way a generalized elliptic curve in the sense of \cite{dr73}, and this construction provides the isomorphism.

Recall that $\MM_\dr(m)(T)$ is the groupoid of flat families of semistable (i.e. each rational component has at least two markings) curves of genus one $E \to T$, such that the singular fibers are N\'eron $m$-gons, together with a group scheme structure on $E^\sm \to T$ making the singular fibers isomorphic to $\Gm \times \Zm$, and an isomorphism $E^\sm[m] \cong (\Zm)^2$. 

There is a canonical map $\MM_{1,1}(m) \to {B(\Zm)^2}$: pulling back the $(\Zm)^2$-torsor over $E^\sm$ along the identity section $T \to E^\sm$ we get a torsor on $T$, and this is clearly functorial. This tells us how to interpret the $2$-fiber product
\[\MM_{1,1}(m) \times_{B(\Zm)^2} \Spec \Z[1/m],\]
which (after writing out the definition) is the stack parametrising elliptic curves with a torsor, together with the added data of a trivialization of the torsor over the identity section. 

\begin{prop} There is an isomorphism 
\[ \MM_\dr(m) \cong \MM_{1,1}(m) \times_{B(\Zm)^2} \Spec \Z[1/m]. \] \end{prop}

\begin{proof} We define mutually inverse functors from both stacks to each other. Start with an object of $\MM_\dr(m)(T)$, so we have a generalized elliptic curve $E \to T$ and a $(\Zm)^2$-action on $E$ coming from the action of $E^\sm$ on $E$. Let $E'$ denote the scheme quotient under this action. The image of the given section of $E$ gives us a section of $E'$. 

We claim that $E' \to T$ is a stable curve of genus one, and $E \to E'$ an admissible cover. On a geometric fiber where $E$ is smooth, $E'$ is also smooth, and $E \to E'$ is \'etale. When $E$ is a N\'eron $m$-gon, $E'$ is a nodal rational curve with ramification index $m$ at the node. Since $E^\sm$, hence also $(\Zm)^2$, acts freely on the smooth locus, the restriction of $E \to E'$ to the smooth locus is a $(\Zm)^2$-torsor. Thus we have an object of $\MM_{1,1}(m)(T)$. Moreover, by construction we also have a lifting of the section $T \to E'$ to a section $T \to E$, which gives us a trivialization of the torsor on $T$ obtained by pulling back $E \to E'$ along $T \to E'$. Clearly this is functorial.

Conversely, an object of $(\MM_{1,1}(m) \times_{ B(\Zm)^2} \Spec \Z[1/m])(T)$ is an elliptic curve $E \to T$, an admissible torsor $P \to E$, and a trivialization of the torsor over the identity section. We claim that $P \to T$ is a semistable curve of genus one whose geometric fibers are either smooth elliptic curves or N\'eron $m$-gons. On the smooth locus this is clear. Over a geometric fiber where $E \to T$ is nodal, we know that the ramification index of $P \to E$ at the node is necessarily $m$ by \cite[6.1.2]{acv03}. Since $P$ is a nodal curve of arithmetic genus one, and $(\Zm)^2$ necessarily acts transitively on the dual graph of the fiber, the only possibility then is that it is a N\'eron $m$-gon. Also, the trivialization of the torsor gives us a section $T \to P$, contained in the smooth locus of $P$. 

We claim that there is a unique structure of generalized elliptic curve on $P \to T$ such that the given section is the identity section, and the action of $(\Zm)^2$ is given by an isomorphism $(\Zm)^2 \cong P^\sm[m]$. Using \cite[II.3.2]{dr73} it suffices to show that $(\Zm)^2$ acts trivially on $\Pic^0_{P/T}$, and using \cite[II.1.7]{dr73} it suffices to show that $(\Zm)^2$ preserves the cyclic ordering of the vertices of the dual graph of $P$. Suppose not: then there is a non-identity $g \in (\Zm)^2$ which maps an irreducible component $C$ of $P$ to itself and interchanges the two nodes of $C$. But then since no automorphism of $\p^1$ is fixed-point free, there is a fixed point in the smooth locus, contradicting that $P^\sm \to E$ is a torsor. \end{proof}

\begin{prop} There is an isomorphism \[ \MM_{1,1}(m) \times_{B(\Zm)^2} \Spec \Z[1/m] \to \MM_{1,1}(m) \!\!\!\fatslash {(\Zm)^2}.\] That is, trivializing the torsor over the identity is the same as rigidifying the stack. \end{prop}

\begin{proof} The forgetful map $\MM_{1,1}(m) \times_{B(\Zm)^2} \Spec \Z[1/m] \to \MM_{1,1}(m)$ composed with the quotient map $\MM_{1,1}(m) \to \MM_{1,1}(m)\!\!\!\fatslash (\Zm)^2$ provides us with the morphism stated in the proposition. To show that it is an isomorphism, it suffices to show that it is a monomorphism and essentially surjective, and the former may be checked on geometric points. Over an algebraically closed field, the automorphism group of a point of the rigidified stack is exactly the quotient of the automorphism group of the original point by the group we are rigidifying along \cite[Thm 5.1.5]{acv03}. Thus it is clear that the map is a monomorphism. To show that it is essentially surjective, we need that for an admissible torsor $P \to C$ over a base scheme $S$, there is locally on $S$ a trivialization of the torsor over the identity. But this is the same thing as trivializing the pullback of $P \to C$ along the identity section $S \to C$, so it comes down to the fact that a torsor on $S$ admits a local trivialization. \end{proof}

\begin{rem} One reason to be interested in this kind of result is that it may help in understanding the reduction of $\MM_{g,n}(BG)$ at primes dividing the order of $G$. Recall that we work over $\Spec \Z[1/m]$ throughout this article, since in \cite{acv03} the stack $\MM_{g,n}(BG)$ is only shown to have any nice properties at all when $|G|$ is invertible on the base. After the publication of that article, the reduction of  $\MM_{g,n}(BG)$ at bad primes has been tentatively studied, for instance in \cite[Section 6]{aov} for the case $G = \mu_2$. However, the reduction of $X(m)$ at primes dividing $m$ is much better understood: the appropriate analogue of level structures that should be used to give a modular interpretation over $\Spec \Z$ on the open part $Y(m)$ was worked out in detail in \cite{katzmazur}. In \cite{conrad05} the boundary is given a modular interpretation as well, and it is shown that $X(m)$ is then a smooth proper Artin stack over $\Spec \Z$. A natural first step for studying the reduction of $\MM_{g,n}(BG)$ may then be to see what the appropriate analogues of the theorems and methods that one has for modular curves are in this particular case. \end{rem}

\begin{rem} As a sanity check, let us study the boundary of $\MM_{1,1}(m)$. We already know from the results of this section that the boundary should consist of a finite set of points, namely $\phi(m)$ times the number of cusps of the curve $X(m)$. Recall from \cite[Section 3.8]{diamondshurman} that the number of cusps of $X(m)$ is 
\[\frac{1}{2} m^2 \prod_{p | m}(1 - \frac{1}{p^2}) \]  
if $m \geq 3$, and $3$ if $m= 2$. Moreover, following their derivation of this formula, the factor 
\[ m^2 \prod_{p | m}(1 - \frac{1}{p^2}) \]  
arises as the number of order $m$ elements of $(\Zm)^2$. 

\newcommand{\tame}{\mathrm{tame}}
Let us see how one may compute this number of cusps also by considering admissible $G$-torsors over a rational curve with a node, where $G = (\Zm)^2$. We work over a separably closed field. Let us first consider admissible torsors over the normalization of the curve. Such torsors correspond to tame covers of $\p^1 \setminus\{x,y\}$ where $x$ and $y$ are two points. Fix an isomorphism 
\[ \pi_1^\tame(\p^1 \setminus\{x,y\}) \cong \widehat \Z \] and a choice $\sigma$ of generator of tame inertia around $x$. Tame $G$-coverings correspond bijectively to homomorphisms $\pi_1^\tame \to G$ by \cite[XIII, 2.12]{sga1}. By \cite[6.1.2]{acv03} $\sigma$ should map to an element of order $m$, of which there are 
\[ m^2 \prod_{p | m}(1 - \frac{1}{p^2}). \]
Note that all of the resulting $G$-covers are disconnected since the homomorphism surely is not surjective. Extend the covering to a branched cover of $\p^1$ using Abhyankar's lemma \cite[Appendice 1, 5.2]{sga1}. The condition that we should obtain a connected admissible torsor over the nodal curve imposes restrictions on how to identify the fibers over the two branch points. Both branch points are torsors for the group $G/H$, where $H$ is the stabilizer, so there are $|G/H| = m$ possible isomorphisms (as torsors) between the fibers. If we choose a global section of the admissible torsor over $\p^1$ and thus get compatible trivializations of the torsors over both branch points, the condition that we should get a connected cover can be expressed by saying that the identity on one fiber should be glued to a generator on the other. There are thus $\phi(m)$ such gluings, and all of them produce non-isomorphic admissible torsors. However, after gluing the branch points together, we can no longer tell the points $x$ and $y$ apart and hence neither $\sigma$ and $\sigma^{-1}$. So we no longer have an element of order $m$ in $G$, only a distinguished \emph{unordered pair} $\{g,g^{-1}\}$. When $m \leq 2$ this makes no difference as $g = g^{-1}$ for all $g \in G$, but for $m \geq 3$ we must divide the number of points on the boundary by two. Thus we obtain exactly the right number of cusps.  \end{rem}

\section{The alternating part of cohomology}
\begin{rem}Throughout this section, $H^\bullet_c$ denotes compactly supported cohomology taking values in either the category of mixed Hodge structures or $\ell$-adic Galois representations. We occasionally write out $\Q$-coefficients or mention the word ''Hodge structure'' for notational convenience, but this should not be interpreted as a preference for either cohomology theory. \end{rem}

\newcommand{\cyc}{{\leftturn}}
\newcommand{\acyc}{{\Circle}}

Let us split the space $\MM_{1,n}(m)$ into three pieces according to the dual graph of the base curve $C$:
\begin{enumerate}
\item the interior $\M_{1,n}(m)$ where the dual graph has a single vertex of genus one;
\item the subspace $\M_{1,n}^{{\acyc}}(m)$ where the dual graph is a \emph{necklace}, i.e. an $N$-cycle of genus zero vertices for some positive integer $N$;
\item the union of all remaining strata. Explicitly, these are all graphs where a non-empty forest of genus zero vertices has been attached to one of the dual graphs appearing in case (1) or (2). 
\end{enumerate}
We write $\M_{1,n}^\acyc$ for $\M_{1,n}^\acyc(1)$. We begin by computing the alternating part of the cohomology of each of these pieces separately. 

If $M$ is any $\sym_n$-module, let $M[\sgn]$ denote the subspace where $\sym_n$ acts by the alternating representation. 

\subsection{Contribution from the interior}

The arguments in the next proposition are similar to those in \cite{faberconsani}, and we omit some details.

\begin{prop} \label{alternating} Let $\pi : {\mathcal E} \to \M_{1,1}(m)$ denote the universal elliptic curve. Let $n > 1$. Then 
\[ H^i_c(\M_{1,n}(m),\Q)[\sgn] \iso \begin{cases} H^1_c(\M_{1,1}(m), \Sym^{n-1} R^1 \pi_\ast \Q) & \text{ if $i =n$,} \\ 0 & \text{ otherwise.}\end{cases} \] \end{prop}

\begin{proof} The proof proceeds in two steps. First, one shows that the natural $\sym_n$-equivariant open embedding
\[ \M_{1,n}(m) \hookrightarrow {\mathcal E}^{n-1}\]
into the $(n-1)$:st fibered power of the universal curve induces an isomorphism $H_c^\bullet(\M_{1,n}(m))[\sgn] \iso H_c^\bullet({\mathcal E}^{n-1})[\sgn]$. From the long exact sequence \cite[Corollary 5.51]{peterssteenbrink}
\[ H^\bullet_c(\M_{1,n}(m)) \to H^\bullet_c({\mathcal E}^{n-1}) \to H_c^\bullet(T) \to H^{\bullet+1}_c(\M_{1,n}(m)), \]where $T$ denotes the complement, it suffices to show that the alternating representation does not occur in the cohomology of $T$. This in turn reduces, by the same exact sequence and inclusion-exclusion, to showing this fact for a subspace of ${\mathcal E}^{n-1}$ defined by two or more points coinciding. But on any such subspace there is a transposition acting trivially, showing that the alternating representation cannot occur in its cohomology. 

Secondly, consider the projection map $\sigma: {\mathcal E}^{n-1} \to \M_{1,1}(m)$. The first observation is that the action of $\sym_n$ maps each fiber to itself: this is clear for the subgroup of the last $n-1$ points, and for the first point one also needs to compose with a translation on the elliptic curve. Hence it makes sense to study the $\sym_n$-action on the complex $R^\bullet \sigma_! \Q = R^\bullet \sigma_\ast \Q$. By computing fiberwise it is seen \cite[Proposition 1]{faberconsani} that the alternating part is concentrated in degree $n-1$ and forms a subspace isomorphic to $ \Sym^{n-1} R^1\pi_\ast \Q$. Moreover, it is known that the local system $\Sym^{n-1} R^1\pi_\ast \Q$ (respectively the smooth $\ell$-adic sheaf $\Sym^{n-1} R^1\pi_\ast \Q_\ell$) has nonvanishing cohomology only in degree $1$. The Leray spectral sequence with compact support for $\sigma$ degenerates at the $E_2$ level by Lieberman's trick or alternatively because $\sigma$ is a smooth projective morphism. Thus one finds that the alternating part of the cohomology is concentrated in degree $n$ and isomorphic to $H^1_c(\M_{1,1}(m), \Sym^{n-1} R^1 \pi_\ast \Q)$, as claimed. \end{proof}

\subsection{Contribution from necklaces}

\renewcommand{\X}{\mathcal X}
In this section we study the alternating part of the cohomology of $\M^{{\acyc}}_{1,n}(m)$ and show that $\chi_c(\M^{{\acyc}}_{1,n}(m))[\sgn]$ has weight zero.

A first observation is that there is a surjection
\[ \M^{{\acyc}}_{1,n}(m) \to \M^{{\acyc}}_{1,1}(m) = \partial \M_{1,1}(m)\]
given by forgetting points, with the property that each fiber is mapped to itself under the $\sym_n$-action, and the fibers are permuted transitively by the action of $\SL(2,\Zm)$ on $\MM_{1,n}(m)$. Hence it is sufficient to consider the cohomology of a single fiber over a cusp of $\partial \M_{1,1}(m)$. Let $\X$ be such a fiber. 

Consider a geometric point $\xi$ of $\X$ corresponding to an admissible torsor $P \to C$. Suppose we forget all but one of the marked points of $C$. Then we may stabilize $C$ to a nodal rational curve $C'$, and as in \cite[Proposition 9.1.1]{av02} we get a unique admissible torsor $P' \to C'$ over the stabilized curve. (The previous sentence just describes the forgetting points-morphism.) Then the admissible torsor $P \to C$ is determined up to isomorphism by the admissible torsor $P' \to C'$ --- that is, which cusp of $\MM_{1,1}(m)$ the point $\xi$ maps to --- and the stabilization map $C \to C'$. This fact is easiest to see using twisted curves, where the corresponding torsor $P \to \mathcal{C}$ is simply the pullback of $P' \to \mathcal{C}'$. 

One is tempted to conclude that the admissible torsor $P \to C$ is determined up to isomorphism by $C$ and the point in $\partial \M_{1,1}(m)$ it maps to under stabilization, and hence that the fiber $\X$ has the same coarse moduli space as $\M_{1,n}^{\acyc}$. This is however not true, as in the previous paragraph the map $C \to C'$ had to be included in the data.

\begin{defn} Let $\M_{1,n}^{\cyc}$ denote the moduli space of stable $n$-pointed curves of genus $1$ whose dual graph forms a necklace, equipped with an orientation on the edges of the dual graph inducing a cyclic ordering of the vertices. \end{defn}

\begin{prop} If $m \leq 2$, then the coarse moduli space of $\X$ is the coarse moduli space of $\M_{1,n}^{\acyc}$. If $m \geq 3$,  $\X$ has coarse moduli space $\M_{1,n}^{\cyc}$. \end{prop}

\begin{proof} Let as before $P \to C$ be a geometric point of $\X$, and $P' \to C'$ the point of $\partial \M_{1,1}(m)$ it maps to under the forgetting points-morphism. 

The nodal rational curve $C'$ has a unique nontrivial automorphism, so given the pointed curve $C$ the map $C \to C'$ is determined up to this involution. Let $x$ be the node of $C'$. Under the isomorphism
\[\pi_1^{\rm tame}(C'\setminus\{{x}\}) \cong \widehat \Z \]
this involution induces multiplication by $-1$. In general an automorphism of $C$ induces the identity on $C'$ or this involution, according to whether the automorphism preserves or reverses the cyclic ordering of the vertices in the necklace. Thus the admissible torsor $P \to C$ is uniquely determined by $C$ precisely if the class of $P'$ in the group
\[\Hom(\pi_1^{\rm tame}(C'\setminus\{x\}),(\Zm)^2) \] is invariant under multiplication by $-1$; otherwise, $P$ is uniquely determined by $C$ once we fix a cyclic ordering of its dual graph. By \cite[6.1.2]{acv03} the class of $P'$ has order $m$ in the group, so it is invariant under $-1$ if and only if $m \leq 2$. 

There is clearly a forgetful morphism from $\X$ to $\M_{1,n}^{\acyc}$ (when $m \leq 2$) respectively $\M_{1,n}^{\cyc}$ (when $m \geq 3$). The discussion above shows that this forgetful morphism is bijective on isomorphism classes. \end{proof}

\begin{rem} While we have an isomorphism of coarse moduli spaces in the previous proposition, the fiber $\X$ is far from being isomorphic to $\M_{1,n}^{\acyc}$ respectively $\M_{1,n}^{\cyc}$. First of all there is the issue of rigidification, but even taking that into account one also has so-called ''ghost automorphisms'' \cite[Section 7]{acv03}. In fact the forgetting points-morphisms are never representable when one works with the spaces $\MM_{g,n}(BG)$. \end{rem}

\begin{rem} An alternative proof of the preceding proposition uses the evaluation maps $\MM_{g,n}(\mathscr X) \to \ri(\mathscr X)$  \cite{abramovich08}. When $\mathscr X = BG$ and $G$ is abelian, the points of $\ri(BG)$ correspond (after a choice of a primitive root) to the elements of $G$, and we can describe the evaluation maps as associating to each marked point the monodromy around the marking of the corresponding admissible torsor. This clarifies the relationship with the computations in $\pi_1^{\rm tame}$ above. One can define evaluation maps not just for markings but also for each branch of a node, and so for every half-edge of the dual graph one gets a decoration by an element of $G$. For instance, Jarvis and Kimura remark that \emph{''the moduli spaces $\MM_{g,n}(BG)$ have boundary strata indexed by stable graphs whose tails and half-edges are decorated by elements of $G$ (up to conjugation)''} in \cite{jarviskimura1}.

It is not hard to show that specifying the decoration of a single half-edge on a necklace determines all of the other decorations, since the product of both elements along an edge should always be the identity in $G$ (otherwise the cover is not admissible), and the product of all elements incident to a genus zero component should also be the identity (by a computation in the [\'etale] fundamental group). By \cite[6.1.2]{acv03} all half-edges of a stratum in $\M^{\acyc}_{1,n}(m)$ are decorated by an element of order $m$, so this decoration is invariant under the dihedral symmetry if and only if $m \leq 2$.  \end{rem}

We now study the two generating series 
\[ \sum_{n=1}^\infty \chi_c(\M^{{\acyc}}_{1,n})[\sgn]t^n \in K_0({\rm MHS}_\Q) \otimes \Q[[t]] \]
and
\[ \sum_{n=1}^\infty \chi_c(\M^{{\cyc}}_{1,n})[\sgn]t^n \in K_0({\rm MHS}_\Q) \otimes \Q[[t]] \]
of compactly supported Euler characteristics, taken in the Grothendieck ring of rational mixed Hodge structures. The reason for passing to Euler characteristics is that we want to compute with each stratum separately, and this is possible since the compactly supported Euler characteristic satisfies the scissor relation
\[ \chi_c(X\setminus Y) =\chi_c(X) - \chi_c(Y) \]
for $Y \subseteq X$ constructible, as well as the usual K\"unneth formula. 

\begin{rem} Given a dual graph $\Gamma$ corresponding to a stratum in $\MM_{g,n}$, the cohomology of the stratum is given by
\begin{equation} \label{cohstrat} \bigotimes_{v \in \rm{Vert} \Gamma} H^\bullet(\M_{g(v),n(v)})_{\Aut \Gamma }\end{equation}
where the subscript denotes coinvariants with respect to the group action. In the case of $\M_{1,n}^{\cyc}$ we must replace $\Aut \Gamma$ with the cyclic subgroup of index two that preserves the cyclic ordering. \end{rem}

The following is proven in \cite{faberconsani}:

\begin{prop} \label{op1} The Euler characteristic $\chi_c(\M^{{\acyc}}_{1,n})[\sgn]$ is pure of weight zero.
\end{prop}

We outline their proof, which uses the theory of modular operads. Let $\Lambda = \bigoplus_n \Lambda^n$ be the ring of symmetric functions. We will not distinguish between $\Lambda^n$ and the representation ring of $\sym_n$, nor between a (virtual) $\sym_n$-representation and a symmetric function of degree $n$.  First the generating function
\[\sum_{n=1}^\infty \chi_c(\M^{{\acyc}}_{1,n}) \in \Lambda \widehat\otimes K_0(\rm{MHS}_\Q)\]
is studied. This sum is naturally interpreted as a sum over graphs, and the so-called semiclassical approximation provides an explicit expression for it. Let $\M$ be the generating function 
\[ \sum_{n=3}^\infty \chi_c(\M_{0,n}) \in \Lambda \widehat\otimes  K_0( {\rm MHS}_\Q). \]
It follows from the main theorem of \cite{semiclassical} (the semi-classical approximation) that
\[ \sum_{n=1}^\infty \chi_c(\M^{{\acyc}}_{1,n}) =\left( -\frac 1 2 \sum_{n=1}^\infty \frac{\phi(n)}{n}\log(1-p_n) \right) \circ \frac{\partial^2 \M}{\partial p_1^2}  + \frac{\frac{\partial \M}{\partial p_2}(2+ \frac{\partial \M}{\partial p_2}) + \frac{\partial^2 \M}{\partial p_1^2}}{4(1-p_2 \circ \frac{\partial^2 \M}{\partial p_1^2})},\]
where $\circ$ denotes plethysm of symmetric functions. Consani and Faber proceed 
to use results of Getzler on the structure of $H^\ast(\M_{0,n})$ as an $\sym_n$-representation to show that only the top-degree cohomology $H^{n-3}_c(\M_{0,n})$, which is pure of weight zero, can contribute nontrivially to the alternating part. 

\begin{rem}In their proof they actually compute an exact expression for $\chi_c(\M^{{\acyc}}_{1,n})[\sgn]$. However, we prefer to deduce this from Theorem \ref{mainthm}. \end{rem}

We claim that similarly
\begin{equation}\label{gummiboll}\sum_{n=1}^\infty \chi_c(\M^{{\cyc}}_{1,n}) = \left( -\sum_{n=1}^\infty \frac{\phi(n)}{n}\log(1-p_n) \right) \circ \frac{\partial^2\M}{\partial p_1^2}.\end{equation}


The author has a combinatorial proof of Getzler's semiclassical approximation which uses wreath product symmetric functions to directly sum over necklaces up to dihedral symmetry. In this proof, the first term of Getzler's formula corresponds to symmetries under rotation, and the second term corresponds to symmetries under reflections. Then the above formulas show that going from dihedral to cyclic symmetry corresponds to discarding all reflection terms (and multiplying by two). This combinatorial proof will be given separately in \cite{semiclassicalremark} since it is not really needed for the purposes of this article. However we give a direct proof of this last formula using little more than the definition of a plethysm of $\sym$-modules. 

\providecommand{\V}{\mathcal{V}}
\providecommand{\W}{\mathcal{W}}
\newcommand{\Ass}{\mathcal{A}ss}
\newcommand{\Alt}{\mathrm{Alt}}

\begin{defn} Let $A$ be a ring. An $\sym$-module $\mathcal V$ over $A$ is the data of an $A[\sym_n]$-module ${\mathcal V}(n)$ for each positive integer $n$.  \end{defn}

\begin{defn} Let $\V$ and $\W$ be $\sym$-modules over $A$. We define their direct sum $\V \oplus \W$ componentwise. We define their tensor product by
\[ (\V \otimes \W)(n) = \bigoplus_{k+l =n} \mathrm{Ind}_{\sym_k \times \sym_l}^{\sym_n} \V(k) \otimes \W(l).\] This makes the category of $\sym$-modules a symmetric monoidal category.  \end{defn}

\begin{defn} Let $\V$ and $\W$ be $\sym$-modules over $A$. The plethysm $\V \circ \W$ is defined by
\begin{equation} \label{plethysm} (\V \circ \W)(n) = \bigoplus_{k=1}^\infty \V(k) \otimes_{A[\sym_k]} (\W^{\otimes k}) (n) \end{equation}
where $(\W^{\otimes k}) (n)$ is considered as an $\sym_k$-module by permuting the factors. \end{defn}

 Let $A$ be the ring $K_0({\rm MHS}_\Q)$. Let $\M$ be the $\sym$-module defined by 
 \[ \M(n) = \chi_c(\M_{0,n})\]
 for $n\geq 3$ and $\M(n)=0$ otherwise. Note that 
\[\left( \frac{\partial^2 \M}{\partial p_1^2}\right)(n) = \mathrm{Res}^{\sym_{n+2}}_{\sym_n} \, \chi_c(\M_{0,n+2}) \]
for $n \geq 1$. Let moreover $\Ass$ denote the $\sym$-module defined by
\[ \Ass (n) = \mathrm{Ind}_{\Z/n\Z}^{\sym_n}\, \mathbf{1}. \] 

\begin{prop} There is an equality \[ \left(\Ass \circ  \frac{\partial^2\M}{\partial p_1^2}\right)(n) = \chi_c(\M_{1,n}^{\cyc}).\]
\end{prop}

\begin{proof} For any $\sym_n$-module $M$ and any subgroup $H$ of $\sym_n$,
\[ \mathrm{Ind}_{H}^{\sym_n}\, \mathbf{1} \otimes_{A[\sym_n]} M = \mathbf{1} \otimes_{A[H]} M= M_{H}, \]
the coinvariants under $H$. Also,  
\[ \left(\frac{\partial^2\M}{\partial p_1^2}\right)^{\otimes k}(n) = \bigoplus_{n_1 + \ldots + n_k = n}  \bigotimes_{i=1}^k \chi_c( \M_{0,n_i+2})\]
by the additivity and multiplicativity of the Euler characteristic. Hence we are done by comparing equations \eqref{cohstrat} and \eqref{plethysm}. \end{proof}

\begin{prop}There is an equality of generating series
\[\sum_{n=1}^\infty \Ass(n) = -\sum_{n=1}^\infty \frac{\phi(n)}{n}\log(1-p_n).\] \end{prop}

\begin{proof}Use the results of \cite[Chapter 1, Section 7, Example 4]{macdonald}. See also \cite[Example 7.6.2]{getzlerkapranov}. \end{proof}

\begin{prop} \label{op2}The alternating part of $\chi_c(\M^{{\cyc}}_{1,n})$ is pure of weight zero.
\end{prop}

\begin{proof}The alternating part of the right hand side of equation \eqref{gummiboll} is shown to have weight zero in  \cite[Lemma 7 and 8]{faberconsani}.
\end{proof}

\begin{rem} The notions of $\sym$-modules and plethysm arise naturally when studying operads, an operad being exactly a monoid in the category of $\sym$-modules with respect to plethysm. In this context the definition of plethysm can be understood graphically. Namely, if $\V$ is an operad, one often thinks of $\V(n)$ as spanned by graphs with one output leg and $n$ input legs, with the $\sym_n$-action corresponding to permutation of the inputs. Then $\V \circ \W$ corresponds to attaching the output legs of the graphs corresponding to $\W$ to the input legs of the graphs corresponding to $\V$ in all possible ways. 

This way of thinking also makes the previous propositions more or less trivial. By our definition of $\Ass(n)$, we can think of it as the $\sym_n$-module spanned by necklaces considered up to cyclic symmetry. To make this picture more operadic, we can replace necklaces by corollas (single vertices with several input legs) by placing a vertex in the middle of the necklace and drawing an input leg from each node of the necklace to the vertex in the middle. Then $\Ass(n)$ is spanned by corollas with $n$ cyclically ordered input legs, or equivalently, by corollas equipped with an embedding in the plane up to ambient isotopy. (In terms of operads, $\Ass$ is the underlying $\sym$-module of the cyclic associative operad shifted by one.)

Also, $\M(n)$ is given by marking $n+2$ points on $\p^1$ and then fixing two of them, so in terms of dual graphs there are two fixed legs and the remaining are permuted by the $\sym_n$-action. To attach $n$ legs at a node of a necklace we need to have $n+2$ marked points on $\p^1$, where the last two vertices are those which we glue to the incident components. Note that the last two vertices are naturally ordered as one is attached to the edge incident in the clockwise direction of the dual graph, and the other is attached counterclockwise.

So $\Ass \circ \M$ is  the result of attaching extra legs to a single cyclically oriented corolla, or equivalently a necklace, which is exactly what we want.  \end{rem}

\subsection{The remaining strata}

\begin{prop} \label{yo} The alternating representation does not occur in the cohomology of any stratum where a non-empty forest of genus zero vertices has been attached to a stable dual graph. \end{prop}

\begin{proof} Let $\Gamma$ be the dual graph of such a stratum of $\MM_{1,n}(m)$. Let $v$ be an extremal vertex of one of the attached trees, say with $N$ incident half-edges. Let $\Gamma'$ be the dual graph given by deleting $v$ and these half-edges. The graph $\Gamma'$ defines a stratum of $\MM_{1,n-N+1}(m)$. If we let $\M(\Gamma)$ and $\M(\Gamma')$ denote the corresponding strata, then there is an isomorphism
\[ \M(\Gamma) \to \M(\Gamma') \times \M_{0,N}. \]
The morphism $\M(\Gamma) \to \M(\Gamma')$ is the morphism that forgets all the points on the component $v$, and the morphism to $\M_{0,N}$ remembers only the markings on the component. This is well-defined since there is at least one marked point on $v$, so no automorphism of the dual graph could interchange the component with any other. To define an inverse, use that any admissible torsor over the component $v$ is necessarily \'etale, and since $\p^1$ is simply connected the torsor is necessarily trivial. Hence given an admissible torsor corresponding to an object of $\M(\Gamma')$ there is a unique extension to an admissible torsor over the attached component $v$ (corresponding to an object of $\M_{0,N}$).

It follows in particular that 
\[H^\bullet(\MM(\Gamma)) = H^\bullet (\MM(\Gamma'))\otimes H^\bullet(\M_{0,N}). \]
Now \cite[Lemma 5]{faberconsani} describes which $\sym_N$-representations may occur in the cohomology of $\M_{0,N}$, and the fact that the alternating representation cannot occur in the left hand side follows by the same Frobenius reciprocity argument as in \cite[Lemma 6]{faberconsani}. \end{proof}

\section{Cusp form motives}

\renewcommand{\V}{\mathbb{V}}
Let $\V_n$ denote the local system $\Sym^{n} R^1\pi_\ast \Q$ on $\M_{1,1}(m)$, where $\pi$ is the projection from the universal elliptic curve. Consider for all $n$ the natural morphisms
\[H^1_c(\M_{1,1}(m),\V_n) \to H^1(\M_{1,1}(m),\V_n).\]
We define the \emph{parabolic cohomology} to be the image of this morphism, and the \emph{Eisenstein cohomology} to be the kernel. The parabolic cohomology is denoted $H^1_!(\M_{1,1}(m),\V_n)$. The weight filtration on $H^1_c(\M_{1,1}(m),\V_n)$ has only two steps,
\[ 0 \subset W_0 \subset W_{n+1} = H^1_c(\M_{1,1}(m),\V_n),\]
where $W_0$ is the Eisenstein cohomology and $W_{n+1}/W_0$ is isomorphic to the parabolic cohomology.  See for instance \cite[\S 4]{faltings87}.

\begin{thm} \label{mainthm} The alternating part of the cohomology of $\MM_{1,n}(m)$ is pure of weight $n$ and coincides with the parabolic cohomology groups 
\[ H^1_!(\M_{1,1}(m), \Sym^{n-1} R^1\pi_\ast \Q)\]
and $H^1_!(\MM_{1,1}(m),\Sym^{n-1} R^1\pi_\ast \Q_\ell)$ in Betti and $\ell$-adic cohomology respectively.\end{thm}

\begin{proof} The proof is similar to \cite[Section 1]{scholl90}. Consider the long exact sequence \cite[Corollary 5.51]{peterssteenbrink}
\[  H^\bullet_c(\M_{1,n}(m)) \to H^\bullet_c(\MM_{1,n}(m)) \to H^\bullet_c(\partial \M_{1,n}(m)) \to H^{\bullet+1}_c(\M_{1,n}(m)).  \]
Take the alternating part of the sequence. $H^\bullet_c(\M_{1,n}(m))[\sgn]$ is concentrated in homological degree $n$ by Proposition \ref{alternating}, so there is an isomorphism 
\[ H^i_c(\MM_{1,n}(m))[\sgn] \to H^i_c(\partial \M_{1,n}(m))[\sgn] \] for all $i$ except $n-1$ and $n$, and a surjection \[H^n_c(\MM_{1,n}(m))[\sgn] \to H^n_c(\partial \M_{1,n}(m)[\sgn].\] Since $\MM_{1,n}(m)$ is a smooth and proper DM-stack, $H^i_c(\partial \M_{1,n}(m))[\sgn]$ is therefore pure of weight $i$ for all $i$ except possibly $i=n-1$. 

Since we know that $\chi_c(\partial \M_{1,n}(m))[\sgn]$ has weight zero, this means that any nonvanishing cohomology in $H^i(\partial \M_{1,n}(m))[\sgn]$ for $i \notin \{0,n-1\}$ must cancel against some cohomology in $H^{n-1}(\partial \M_{1,n}(m))[\sgn]$. But  $\partial \M_{1,n}(m)$ is proper, so its $i$th cohomology has weight at most $i$ for all $i$. Hence 
\[H^i_c(\partial \M_{1,n}(m))[\sgn] = 0 \text{ for all $i \geq n$.}\] Moreover, since $H^i_c(\partial \M_{1,n}(m))[\sgn] \cong H^i_c(\MM_{1,n}(m))[\sgn]$ for $i \notin \{n-1,n\}$, Poincar\'e duality for $\MM_{1,n}(m)$ shows that $H^i_c(\partial \M_{1,n}(m))[\sgn] $ vanishes also for $i < n-1$. 

Thus the alternating part of the above long exact sequence is concentrated in the short exact sequence
\[ 0 \to H^{n-1}_c(\partial \M_{1,n}(m))[\sgn] \to H^n_c(\M_{1,n}(m))[\sgn] \to H^n_c(\MM_{1,n}(m))[\sgn] \to 0. \]
Now $H^n_c(\M_{1,n}(m))[\sgn] \cong H^1_c(\M_{1,1}(m),\V_{n-1})$ by Proposition \ref{alternating} and we know the weight filtration on this cohomology group by the remarks preceding this proposition. Since we also know the weights on the other two spaces in the sequence, the only possibility is that the first map is the inclusion of the Eisenstein cohomology and the second map is the projection to the parabolic cohomology. \end{proof}

As first observed by To\"en \cite{toenmotives}, any smooth and projective DM-stack $\X$ of finite type over a field $k$ defines a Chow motive over $k$. In the special case of $\X = \MM_{g,n}(BG)$, where $G$ is a finite group, one can also argue as follows: Combining the results of \cite[Section 4]{globalquotient} and \cite[Theorem 2.1]{kreschvistoli}, we get a finite morphism $f \colon X \to \X$ of degree $m$ where $X$ is a smooth and projective variety over $k$. We may then define $h(\X)$ to be the Chow motive defined by $X$ and the projector $\frac{1}{m}[f^\ast] \circ [f_\ast]$.

The stack $\MM_{1,n}(m)$ is the disjoint union of $\phi(m)$ open and closed substacks, each of which corresponds to symplectic level $m$ structure. Let $\MM_{1,n}^s(m)$ be one of these open and closed substacks.  

\begin{thm} \label{hejsan} The Chow motive defined by $\MM_{1,n}^s(m)$ and the projector onto the alternating representation is the motive associated to cusp forms of weight $n+1$ for $\Gamma(m)$. \end{thm}

\begin{proof}The results of Section 3 of this paper identifies the modular curve $X(m)$ with the rigidification of  $\MM_{1,1}^s(m)$, and similarly for $Y(m)$ and $\M_{1,1}^s(m)$. The $\ell$-adic Galois representations classically associated to cusp forms of level $m$ and weight $n+1$ are the parabolic cohomology groups
\[  H^1_!(Y(m),\Sym^{n-1}R^1\pi_\ast \Q_\ell). \]
It remains to identify the parabolic cohomology group on $\M_{1,1}^s(m)$ and $Y(m)$ with each other. Let $\rho : \M_{1,1}^s(m) \to \M_{1,1}^s(m)\!\!\! \fatslash (\Zm)^2 \cong Y(m)$ be the rigidification. The group $(\Zm)^2$ acts trivially on the universal elliptic curve over $\M_{1,1}^s(m)$ and hence also on $\Sym^{n-1}R^1\pi_\ast \Q_\ell$, so $R^0\rho_\ast$ is fiberwise an isomorphism of local systems while all higher direct images vanish. Hence by the Leray spectral sequence the two cohomology groups coincide.  \end{proof}

\begin{rem}By assuming that $n > 1$ in the definition of the Chow motive $M$ we have ruled out cusp forms of weight $2$. These motives are instead well understood in terms of Jacobians of modular curves. \end{rem}

\section{Hecke operators}
Let $p$ and $m$ be positive coprime integers. Throughout this section, we denote by $G$ the group $(\Zm)^2$, and by $H$ the group $\Zp$.  

\begin{defn} Let $\MM_{1,n}(m;p)$ be the open and closed substack of 
\[ \MM_{1,n}(BG \times BH) \times_{BH} \Spec \Z[1/pm]\]
where both the $G$- and $H$-admissible torsors are connected and unramified at the marked points. The fiber product is taken over the evaluation map at the first marked point, that is, the $H$-torsor is trivialized over that marking. \end{defn}

This stack will be used to define an $\sym_n$-equivariant correspondence from $\MM_{1,n}(m)$ to itself, such that when $p$ is a prime, the induced endomorphism of the realization of the cusp form motive is exactly the Hecke operator $T_p$. 

We define two maps 
\[ \phi, \psi : \MM_{1,n}(m;p) \to \MM_{1,n}(m).\]
The map $\phi$ is induced from the obvious map $BG \times BH \to BG$. Concretely, we just forget the $H$-admissible torsor. 

To define the map $\psi$, consider an object of $\MM_{1,n}(m;p)(S)$. Let $P_G \to C$ and $\pi: P_H \to C$ be the two admissible torsors. Since there is a distinguished lifting of the identity section to $P_H$ it has a canonical structure of a generalized elliptic curve (cf. the arguments of Section \ref{comparesection}) and $P_H^\sm \to C^\sm$ is a morphism of group schemes over $S$. One gets a dual isogeny $\pi^\vee : C^\sm \to P_H^\sm$ defined by the property that $\pi^\vee \circ \pi = [p]$.  Thus the remaining sections $S \to C$ also get canonical lifts to $P_H$. 

Moreover, we may pull back $P_G \to C$ to an admissible cover $\pi^\ast P_G \to P_H$, unramified away from the nodes. To see this, it is most convenient to work instead with the associated \emph{twisted} curves ${\mathcal P}_G, {\mathcal P}_H$ and ${\mathcal C}$. Then  ${\mathcal P}_G$ and ${\mathcal P}_H$  are genuine torsors over ${\mathcal C}$ and the pullback $\pi^\ast P_G$ is just the ordinary fibered product ${\mathcal P}_H \times_{\mathcal C} {\mathcal P}_G$. By the argument of \cite[Lemma 2.2.1]{acv03}, $\pi^\ast P_G$ is untwisted. Moreover, $\pi^\ast P_G$ is a connected curve: since the groups $H$ and $G$ have coprime exponents by assumption, the surjectivity of the maps
\[ \pi_1({\mathcal C}) \to H, \qquad \pi_1({\mathcal C}) \to G, \]
implies the surjectivity of the product map to $H \times G$. 

Hence $\pi^\ast P_G \to P_H$ is an $n$-pointed, in general only prestable, curve with a connected admissible $G$-torsor, and by stabilization \cite[Proposition 9.1.1]{av02} we get an object of $\MM_{1,n}(m)(S)$. This defines the map $\psi$.

The argument above also shows how to define an $\sym_n$-action on $\MM_{1,n}(m;p)$. It is clear how the subgroup permuting the last $n-1$ points acts. For a permutation switching the first and $i$:th point, we use the dual isogeny induced by the trivialization of the $H$-torsor over the identity section to get a trivialization over the $i$:th section as well. With this definition, both $\phi$ and $\psi$ are $\sym_n$-equivariant. 

\begin{thm} Let $p$ be a prime. The composition $[\phi_\ast] \circ [\psi^\ast]$ defines an endomorphism of the Chow motive $M$. The induced endomorphism of the Betti realization coincides with the direct sum of the classical Hecke operator $T_p$ acting on holomorphic cusp forms and its conjugate acting on antiholomorphic cusp forms, and the induced endomorphism of the $\ell$-adic realizations is the $\ell$-adic analogue of the Hecke operator, satisfying the Eichler-Shimura relation. \end{thm}

\begin{proof} Like \cite{scholl90}, we show that the induced endomorphism on the open part 
\[ H^\bullet(\M_{1,n}(m))[\sgn] \] 
coincides with the action of the Hecke operator on both cusp forms and Eisenstein series. It may be helpful to compare this proof with the proof of Proposition \ref{alternating}. Consider first the diagram
\[ \begin{diagram} 
\M_{1,n}(m) & \lTo^\phi & \M_{1,n}(m;p) & \rTo^\psi & \M_{1,n}(m) \\
\dTo & & \dTo & & \dTo \\
\M_{1,1}(m) & \lTo^{\phi'} & \M_{1,1}(m;p) & \rTo^{\psi'} & \M_{1,1}(m). \end{diagram} \]
By allowing markings to coincide we get an open embedding of all the spaces on the top row into fibered powers of the respective universal elliptic curves over the spaces on the bottom row, i.e. 
\[ \begin{diagram} 
{\mathcal E}^{n-1} & \lTo^\phi & {\mathcal E'}^{n-1} & \rTo^\psi & {\mathcal E}^{n-1} \\
\dTo & & \dTo & & \dTo \\
\M_{1,1}(m) & \lTo^{\phi'} & \M_{1,1}(m;p) & \rTo^{\psi'} & \M_{1,1}(m), \end{diagram} \]
and one checks that the definition of $\phi$ and $\psi$ makes sense in exactly the same way also on these bigger spaces, and that $\phi$ and $\psi$ are equal to the $(n-1)$st fibered power of the morphism between universal curves induced by $\phi'$ and $\psi'$. 

By the same arguments as in the proof of Proposition \ref{alternating}, one finds that: (i) these open embeddings induce an isomorphism on the alternating part of the cohomology; (ii) the degeneration of the Leray spectral sequence implies that the alternating part of the cohomology of each space on the top row is given by the cohomology of the local system $\Sym^{n-1} R^1\pi_\ast \Q$ on each space on the bottom row, where $\pi$ is the projection from the universal curve. By the functoriality of the Leray spectral sequence, the endomorphism on $H^\bullet(\M_{1,n}(m))[\sgn]$ induced by $[\phi_\ast][ \psi^\ast]$ is the same as the one induced on 
\[ H^1(\M_{1,1}(m), \Sym^{n-1}R^1\pi_\ast \Q)\]
by $\phi'$ and $\psi'$. As in Section \ref{comparison1} we may rigidify $\M_{1,1}(m)$ and $\M_{1,1}(m;p)$ with respect to the action of $(\Zm)^2$; doing so, one finds by arguments very similar to those of Section \ref{comparison1} that the resulting spaces are isomorphic to the spaces called $M_m$ resp. $M_{m,p}$ in \cite{deligne69}, and that the maps $\phi'$ and $\psi'$ coincide with the morphisms $q_1$ and $q_2$ defined in \cite[equation 3.14]{deligne69}. As in Theorem \ref{hejsan} rigidification induces an isomorphism on cohomology, and the proof follows by \cite[Proposition 3.18]{deligne69}.  \end{proof}

\begin{rem} As done by \cite{scholl90}, the Hecke operators can be used to decompose the cusp form motives. Let $f$ be a normalized newform of some weight and level. We would like to associate a motive to $f$, which should be a submotive of the motive $M$ we have already associated to the space of all cusp forms of this weight and level. We consider now $M$ as a Chow motive modulo \emph{homological} equivalence (say with respect to Betti cohomology) instead, and consider the subalgebra $H$ of ${\rm End}\,M$ generated by all $T_p$. This subalgebra is semisimple because its image under the Hodge realization is semisimple. The newform $f$ lies in a $1$-dimensional eigenspace for all the $T_p$, and this eigenspace is a simple $H$-submodule, so there exists an idempotent in the algebra $H$ whose image is this eigenspace. This idempotent defines a Chow motive modulo homological equivalence which is associated to $f$. \end{rem}

\begin{rem} Another way of decomposing cusp form motives into smaller pieces comes from looking also at the congruence subgroups $\Gamma_1(m)$ and $\Gamma_0(m)$. One may define a space of pointed stable curves with $\Gamma_1(m)$-level structure in much the same way as we have done in this article by considering the moduli spaces \[ \MM_{1,n}(B\Zm) \] instead of $B(\Zm)^2$. Explicitly, we look at the open and closed substack consisting of connected admissible torsors which are unramified over each marked point. The arguments in this article go through with only very minor changes. One finds by arguments similar to those in Section 3 that for $n=1$ the curve $X_1(m)$ is recovered, and just as in this article it is seen that by projecting onto the alternating representation of $\sym_n$ one isolates the part of the cohomology associated to cusp forms. Moreover, there is an action of $(\Zm)^\ast$ on the space, under which the cohomology decomposes into pieces indexed by the characters of $(\Zm)^\ast$. This way one can construct also motives associated to spaces of cusp forms of given level, weight and nebentypus. This has the advantage over the construction with Hecke operators that it does not take us out of the category of Chow motives modulo rational equivalence. \end{rem}

\section*{Acknowledgement}

The author gratefully acknowledges support by the G\"oran Gustafsson Foundation for Research in Natural Sciences and Medicine and the G. S. Magnuson Foundation.

\bibliographystyle{alpha}

\bibliography{../database}


\end{document}